\documentclass{kluwer}
\usepackage{amsmath}
\usepackage{amsthm,graphicx,klucite}
\usepackage[margin=3cm]{geometry}

\newtheorem{theorem}{Theorem}[section]

\newtheorem{proposition}{Proposition}[section]

\theoremstyle{remark}
\newtheorem{remark}{Remark}[section]
\theoremstyle{definition}

\newcommand{\set}[1]{\left\{#1\right\}}

\newcommand{\ind}[1]{\mathbf{1}_{#1}}
\newcommand{\abs}[1]{\left\vert#1\right\vert}
\newcommand{\pr}{\mathsf P}
\newcommand{\ex}[1]{\mathsf{E}\left[\,#1\,\right]}

\newcommand{\R}{{\mathbb R}}

\newcommand{\eps}{\varepsilon}

\newcommand{\la}{\lambda}
\newcommand{\lt}{\left}
\newcommand{\rt}{\right}
\newcommand{\var}{\mathsf{var}\,}
\begin{document}

\begin{opening}
\title{Asymptotic behavior of mixed power variations and statistical estimation in mixed models\thanks{This work has been partially supported by the Commission of the European Committees Grant PIRSES-GA-2008-230804 within the program ``Marie Curie Actions''. The authors are also grateful to Rim Touibi for her careful reading of the manuscript.}}
\runningtitle{Asymptotic behavior of mixed power variations and statistical estimation in mixed models}

\author{Marco \surname{Dozzi}\email{marco.dozzi@univ-lorraine.fr}}
\institute{
Institut Elie Cartan\\
Universit\'e de Lorraine\\
Campus-des-Aiguillettes\\
BP 70239\\
F-54506 Vandoeuvre-les-Nancy Cedex\\
France
}

\author{Yuliya  \surname{Mishura}\email{myus@univ.kiev.ua}}
\author{Georgiy \surname{Shevchenko}\email{zhora@univ.kiev.ua}}%}

\institute{National Taras Shevchenko University of Kyiv\\
Department of Probability Theory, Statistics and Actuarial Mathematics\\
Volodymirska 60\\
01601 Kyiv
Ukraine}

\begin{abstract}
We obtain results on both weak and almost sure asymptotic behaviour of power variations of a linear combination of independent Wiener process and fractional Brownian motion. These results are used to construct strongly consistent parameter estimators in mixed models.
\end{abstract}

\classification{MSC 2010}{60G22, 62M09, 60G15, 62F25}
\end{opening}

\section{Introduction}
A fractional Brownian motion (fBm) is frequently used to model short- and long-range dependence. By definition, an fBm with Hurst parameter $H\in(0,1)$ is a centered Gaussian process $\set{B^H_t,t\ge 0}$ with the covariance function
\begin{equation*}
\ex{B_t^HB_s^H} = \frac12\left(t^{2H}+s^{2H}-\abs{t-s}^{2H}\right). 
\end{equation*}
For $H>1/2$, an fBm has a property of long-range dependence; for $H<1/2$, it is short-range dependent and, in fact, is counterpersistent, i.e.\ its increments are negatively correlated. For $H=1/2$, an fBm is a standard Wiener process.

Two important properties of an fBm are the stationarity of increments and self-similarity. However, these properties restrict applications of an fBm, and recently so-called multifractional processes gained huge attention. Multifractionality can consist both in dependence of memory depth and regularity of process on the time instance and on the time scale. In this paper, we are dealing with the latter kind of multifractionality, where the properties of a process depend on the size of the time interval, on which the process is considered. In other words, we are considering processes, which are not inherently self-similar. A simplest approach is to consider a linear combination of independent fBms with different Hurst parameters. 

Here we will concentrate on the case where we have only two fBms and one of them has Hurst parameter equal to $1/2$, simply put, it is a Wiener process. So we consider a process 
\begin{equation} \label{MtH}
M_t^H = a B_t^H + b W_t, t\ge 0
\end{equation}
where $a$ and $b$ are some non-zero coefficients. Such mixed models and their applications where considered in many papers, see \cite{andromis,cheredito,filatova,misshev,bookmyus}. 

The main aim of this paper is statistical identification of model \eqref{MtH}, i.e.\ the statistical estimation of the model parameters. The principal attention will be given to the estimation of $H$, though we will also present estimators for $a$ and $b$. Our secondary goal is to study both weak and almost sure asymptotic behaviour of mixed power variations.

In the ``pure'' fBm case, there exist several methods to estimate the Hurst parameter, an extensive overview of which is given in \cite{coeurjolly}. The most popular methods are based on quadratic and, more generally, higher power variations of the process. A huge literature is devoted to such questions, we will cite only few: asymptotic behaviour of power variations and, more generally, of non-linear transformations of stationary Gaussian sequences is studied in \cite{breuermajor,dobrushinmajor,giraitissurgailis,taqqu}, and stochastic estimation for fBm and multifractional processes with the help of power variations, in \cite{benassi,coeurjolly2,coeurjolly3,girrobsur,istaslang,taqqutev}. Weighted power variations serving similar purposes for stochastic differential equations driven by fBm, were studied in \cite{nourdin,nourdinnularttudor}.

The only papers concerned with parameter estimation in the mixed model are \cite{kleptsyna,kozamis,kitaycy, filatova}, but they address questions different from the one we are interested in.  Namely, we aim at estimating the parameters of the process \eqref{MtH} based on its single observation on a uniform partition of a fixed interval. To this end, we use power variations of this process. 
We remark that, in contrast to the pure fractional case, there is no self-similarity property in the mixed model \eqref{MtH}, so we cannot directly apply the results of \cite{breuermajor,dobrushinmajor,giraitissurgailis,taqqu} on the asymptotic behaviour of sums of transformed stationary Gaussian sequences. For this reason we need to study the asymptotic behaviour as $n\to\infty$ of ``mixed'' power variations of the form 
\begin{equation*} %\label{mixed-var}
\sum_{i=0}^{n-1} \left(W_{(i+1)/n}-W_{i/n}\right)^p\left(B^H_{(i+1)/n}-B^H_{i/n}\right)^r,
\end{equation*}
involving increments of independent fBm $B^H$  and Wiener process $W$, where $ p\ge 0, r\ge 0$ are fixed integer parameters. For statistical purposes, in order to construct strongly consistent estimators, we need mainly the almost sure behavior of the power variations. However, we also study their weak behavior, which is of independent interest.

The paper is organized as follows. Section~\ref{mixedvarsec} contains results on the asymptotic behaviour of mixed power variations.  These results are used in Section~\ref{statsec} to construct strongly consistent estimators of parameters $H,a,b$ in model \eqref{MtH} and study asymptotic normality of the estimators of $H$. 
Finally, in Section~\ref{simsec} we present simulation results to illustrate quality of the estimators provided.

\section{Asymptotic behaviour of mixed power variations}\label{mixedvarsec}

Let $W=\{W_t, t\ge 0\}$ be a standard Wiener process and $B^H=\{B^H_t, t\ge 0\}$ be an independent of $W$ fBm with Hurst parameter $H\in(0,1)$ defined on a complete  probability space $(\Omega, \mathcal{F}, P)$.  

For a function $X\colon[0,1]\rightarrow \R$ and integers $n\ge 1$, $i=0,1,\dots,n-1$ we denote $\Delta_i^n X= X_{(i+1)/n}-X_{i/n}$. 
In this section we will study the asymptotic behavior as $n\to\infty$ of the following mixed power variations
$$
\sum_{i=0}^{n-1} \left(\Delta_i^n W\right)^p\left(\Delta_i^n B^H\right)^r,
$$
where $p\ge 0$, $r\ge 0$ are fixed integer numbers. Thanks to self-similarity of $B^H$ and $W$, the sequence $\set{(n^{1/2} \Delta_i^n W,n^{H}\Delta_i^nB^H),0\le i\le n-1}$ is equivalent in law to $\set{(\xi_i, \zeta_i),0\le i\le n-1}$, where $\set{\xi_j,j\ge 0}$ is a sequence of independent standard Gaussian random variables,  $\set{\zeta_j,j\ge 0}$ is an independent of $\set{\xi_j,j\ge 0}$ stationary sequence of standard Gaussian variables with the covariance  $$\rho_H(m)=\ex{\zeta_0 \zeta_m}= \ex{B_1^H (B_{m+1}^H - B_m^H)}=\frac12\lt(\abs{m+1}^{2H}+\abs{m-1}^{2H}\rt)
-\abs{m}^{2H}.
$$
Therefore, by the ergodic theorem, 
\begin{equation*} %\label{ergodic}
n^{rH+p/2-1}\sum_{i=0}^{n-1} \left(\Delta_i^n W\right)^p\left(\Delta_i^n B^H\right)^r\to \mu_p\mu_r,\ n\to\infty,
\end{equation*}
a.s., where for an integer $m\ge 0$
$$
\mu_m =\ex{N(0,1)^m} = (m-1)!!\ind{m\text{ is even}}
$$
is the $m$th moment of the standard Gaussian law. 
So it is natural to study  centered sums of the form
$$
S^{H,p,r}_n=\sum_{i=0}^{n-1} \left(n^{rH+p/2}\left(\Delta_i^n W\right)^p\left(\Delta_i^n B^H\right)^r
- \mu_p\mu_r \right).
$$

The following theorem summarizes the  limit behaviour of $S_n^{H,p,r}$. We remark that some (but not all) of the results can be obtained from the limit theorems for stationary Gaussian sequences of vectors, see e.g.\ \cite{arcones}. However, we believe that our approach (using one-dimensional limit theorems) is more accessible and leads quicker to the desired results.
\begin{theorem} If $p$ and $r$ are even, $r\ge 2$, then 
\begin{itemize}
\item for $H\in(0,3/4)$
\begin{equation}\label{rpevenH<3/4}
n^{-1/2}{S_n^{H,p,r}}  \Rightarrow N(0,\sigma^2_{H,r}\mu^2_p+\sigma^2_{p,r}),\ n\to\infty,
\end{equation}
where
\begin{gather*}
\sigma^2_{H,r} =  \sum_{l=1}^{r/2}\frac{(l!)^2}{(2l)!((r-2l)!!)^2}
\sum_{m=-\infty}^{\infty}\rho_H(m)^{2l},\quad
\sigma^2_{p,r} = \mu_{2r}\left(\mu_{2p}-\mu_p^2\right);
\end{gather*}
\item for $H=3/4$
\begin{equation}\label{rpevenH=3/4}
\frac{S^{3/4,p,r}_n}{\sqrt{n\log n}}\Rightarrow N(0,\sigma^2_{3/4,r}\mu^2_p+\sigma^2_{3/4,r}),\ n\to\infty,
\end{equation}
where
$\sigma_{3/4,r} = 3r(r-1)/4$;
\item for $H\in(3/4,1)$
\begin{equation}\label{rpevenH>3/4}
n^{1-2H}S_n^{H,p,r} \Rightarrow \zeta_{H,p,r},\ n\to\infty,
\end{equation}
where $\zeta_{H,p,r}$ is a special ``Rosenblatt'' random variable.
\end{itemize}

If $p$ is odd, then for any $H\in(0,1)$
\begin{equation}\label{podd}
n^{-1/2}S_n^{H,p,r} \Rightarrow N(0,\mu_{2p}\mu_{2r\vphantom{p}}).
\end{equation}

If $p$ is even and $r$ is odd,  then 

\begin{itemize}
\item for $H\in(0,1/2]$
\begin{equation}\label{roddH<1/2}
n^{-1/2} S_n^{H,p,r} \Rightarrow N(0,\sigma^2_{H,r}\mu_p^2+\sigma^2_{p,r}),\ n\to\infty,
\end{equation}
where $\sigma_{H,1}=0$, 
$$\sigma^2_{H,r} =  \sum_{l=1}^{(r-1)/2}\frac{(r!)^2}{(2l+1)!((r-2l-1)!!)^2}
\sum_{m=-\infty}^{\infty}
\rho_H(m)^{2l+1},\quad r\ge 3;
$$
\item for $H\in(1/2,1)$
\begin{equation}\label{roddH>1/2}
n^{-H} S_n^{H,p,r} \Rightarrow N(0,\mu_p^2\mu^2_{r+1}),\ n\to\infty.
\end{equation}
\end{itemize}\end{theorem}
\begin{remark}
For $r=0$ we have the pure Wiener case, so for any $H\in(0,1)$
$$
n^{-1/2} S_n^{H,p,r}\Rightarrow N(0,\mu_{2p}-\mu_p^2), \quad n\to\infty.
$$
Also note that in the case $p=0$, $r=1$ the limit variance in \eqref{roddH<1/2} vanishes. Obviously, in this case
$$
n^{-H} S_n^{H,0,1} = B^H_1,
$$
so it has the standard normal distribution.
\end{remark}
\begin{proof}
We study different cases in the same order as they appear in the formulation.

Assume first that $p$ and $r\ge 2$ are even. The principal idea in this case is to rewrite mixed power variation as 
$$
S^{H,p,r}_n = S'_n + S''_n,
$$
where
\begin{equation*}
\begin{gathered}
S_n' = n^{rH}\sum_{i=0}^{n-1} \left(\Delta_i^n B^H\right)^r\left(n^{p/2}\left(\Delta_i^n W\right)^p - \mu_p\right),
\\
S_n'' = \mu_p\sum_{i=0}^{n-1} \lt(n^{rH}\lt(\Delta_i^n B^H\rt)^r
-\mu_r\rt).
\end{gathered}
\end{equation*}
Then we apply known results concerning asymptotic behaviour of $S_n''$, since it contains only fractional Brownian motion, and consider  $S_n'$ conditionally on the fractional Brownian motion. Further we realize this idea.

For $H\in(0,3/4)$, write
\begin{gather*}
n^{-1/2}S^{H,p,r}_n= A_n' + A_n''
\end{gather*}
with $A_n' = n^{-1/2}S_n', A_n'' = n^{-1/2}S_n''$.
According to \cite{breuermajor}, for $r$ even, $H\in(0,3/4)$,
\begin{equation}
\label{sn2}
n^{-1/2}\sum_{i=0}^{n-1} \lt(n^{rH}\lt(\Delta_i^n B^H\rt)^r-\mu_r\rt)\Rightarrow N(0,\sigma^2_{H,r}),
\end{equation}
as $n\to\infty$.
Consequently, $A_n''\Rightarrow N(0,\sigma^2_{H,r}\mu_p^2)$, $n\to\infty$.
Further,
\begin{gather*}
A_n' = n^{rH-1/2} R_n k_n,
\end{gather*}
where
\begin{equation*} %\label{Rnkn}
\begin{gathered}
R_n=\frac{1}{k_n}\sum_{i=0}^{n-1} \la_{k,n}\left(n^{p/2}\left(\Delta_i^n W\right)^p -\mu_p\right)\\
k_n = \lt((\mu_{2p}-\mu_p^2)\sum_{i=0}^{n-1} \la_{i,n}^{2}\rt)^{1/2},\ \la_{i,n}= \lt(\Delta_i^n B^H\rt)^r.
\end{gathered}
\end{equation*}
Since $B^H$ is uniformly continuous a.s.,  $\max_{1\le k\le n}\la_{k,n}\to 0, n\to\infty$ a.s.
Thus, taking into account independence of $B^H$ and $W$ and an evident fact that $\var R_n = 1$, we get by CLT that the  conditional distribution  of $R_n$ given $B^H$ converges to standard normal distribution as $n\to\infty$ a.s.
Further, from the ergodic theorem
$$
n^{2rH-1}\sum_{i=0}^{n-1}(\Delta_i^n B^H)^{2r}\to \mu_{2r},\ n\to\infty,
$$
a.s., hence
$$
n^{rH-1/2}k_n\to  \lt(\mu_{2r}(\mu_{2p}-\mu_p^2)\rt)^{1/2}, n\to\infty \text{ a.s.}
$$
So by  Slutsky's theorem, the conditional distribution of $A_n'$ given $B^H$ converges to $N(0,\sigma^2_{p,r})$ a.s., that is, for any $t\in\R$
we have
\begin{equation}
\int_\R e^{\imath tx}\pr(A_n'\in dx\mid B^H) \to e^{-t^2\sigma^2_{p,r}/2}\label{cfconv}
\end{equation}
a.s.\ as $n\to \infty$.
Now write
\begin{gather*}
\ex{e^{it(A_n' + A_n'')}}  = \ex{\ex{e^{\imath t(A_n' + A_n'')}\middle|B^H}}=\ex{\int_\R e^{\imath tx}\pr\lt(A_n'\in dx\middle| B^H\rt)e^{itA_n''}},
\end{gather*}
whence
\begin{gather*}
\abs{\ex{e^{\imath t(A_n' + A_n'')}}- e^{-t^2(\sigma^2_{H,r}\mu_p^2+\sigma^2_{p,r})/2}}\le E_1+E_2,
\end{gather*}
where
\begin{gather*}
E_1
%= \abs{\ex{\lt(\int_\R e^{itx}\pr\lt(S_n'\in dx\mid B^H\rt)-e^{-t^2 \sigma^2_{p,r}}\rt) e^{itS_n^2}}}\\
= \abs{\ex{\lt(\int_\R e^{\imath tx}\pr\lt(A_n'\in dx\middle| B^H\rt)-e^{-t^2 \sigma^2_{p,r}/2}\rt) e^{\imath tA_n''}}}\to 0,\ n\to\infty
\end{gather*}
by \eqref{cfconv} and dominated convergence;
\begin{gather*}
E_2 = e^{-t^2 \sigma^2_{p,r}/2}\abs{\ex{ e^{\imath tA_n''}}-e^{-t^2\sigma^2_{H,r}\mu_p^2/2}}\to 0,\ n\to\infty\end{gather*}
by \eqref{sn2}. It follows that $$n^{-1/2}{S_n^{H,p,r}}  \Rightarrow N(0,\sigma^2_{H,r}\mu^2_p+\sigma^2_{p,r}),\ n\to\infty,$$
as required in this case.

In the case where $H=3/4$, we have by \cite{breuermajor}
$$\frac{1}{\sqrt{n\log n}}\sum_{k=0}^{n-1} \lt(n^{rH}\lt(\Delta_k^n B^H\rt)^r-\mu_r\rt)\Rightarrow N(0,\sigma^2_{3/4,r}),
$$
as $n\to\infty$, whence \eqref{rpevenH=3/4} can be deduced using the same reasoning as above.

For $H\in(3/4,1)$, write
\begin{gather*}
n^{1-2H}S^{H,p,r}_n =n^{1-2H}S_n' + n^{1-2H}S_n''
=n^{1-2H+rH}R_n k_n + n^{1-2H}S_n'',
\end{gather*}
where $R_n,k_n,S_n',S_n''$ are defined above. As before, $R_n\Rightarrow N(0,1)$ conditionally given $B^H$ as $n\to\infty$ a.s.  However, this time $n^{1-2H+rH}k_n\to 0$, $n\to\infty$ a.s., since $n^{rH-1/2} k_n$ has a finite limit and $n^{3/2-2H}\to 0$, $n\to\infty$. Therefore, $n^{1-2H}S_n'\to 0$, $n\to\infty$. Further, according to \cite{dobrushinmajor}, see also \cite{giraitissurgailis,taqqu},
$$
n^{1-2H}S_n'' \Rightarrow \zeta_{H,p,r},\ n\to\infty,
$$
where $\zeta_{H,p,r}$ is a ``Rosenblatt'' random variable. 
Thus,  we get \eqref{rpevenH>3/4} using Slutsky's theorem. This finishes the case where $p$ and $r$ are even.

Now assume that $p$ or $r$ is odd.
In this case $S_n^{H,p,r}$ has a form
$$
S^{H,p,r}_n=\sum_{k=0}^{n-1} n^{rH+p/2}\left(\Delta_k^n W\right)^p\left(\Delta_k^n B^H\right)^r.
$$
Write
$$
n^{-1/2}S^{H,p,r}_n = n^{rH-1/2}R_n k_n + Z_n,
$$
where
$R_n$ and $k_n$ are defined above,
$Z_n = n^{rH-1/2}\mu_p\sum_{k=0}^{n-1} \lt(\Delta_k^n B^H\rt)^r$.
As before, given $B^H$, $n^{rH-1/2}R_n k_n\Rightarrow N(0,\sigma_{p,r}^2)$, $n\to\infty$, a.s.

Now if $p$ is odd, we have $Z_n=0$ irrespective of the value of $H$, 
whence \eqref{podd} immediately follows.

Further, assume that $p$ is even and $r\ge 3$ is odd. For $H\in(0,1/2]$,
we have by \cite{breuermajor}
$$
n^{rH-1/2}\sum_{k=0}^{n-1} \lt(\Delta_k^n B^H\rt)^r\Rightarrow N(0,\sigma^2_{H,r}),\ n\to\infty.
$$
Therefore, $Z_n \Rightarrow N(0,\mu_p^2\sigma^2_{H,r})$, $n\to\infty$. Arguing as in deriving of \eqref{rpevenH<3/4}, we get
\eqref{roddH<1/2}. For $H\in(1/2,1)$, it follows from \cite{dobrushinmajor} that
\begin{equation*} %\label{roddH>1/2AN}
n^{(r-1)H}\sum_{k=0}^{n-1} \lt(\Delta_k^n B^H\rt)^r\Rightarrow N(0,\mu^2_{r+1}),\ n\to\infty,
\end{equation*}
so $n^{1/2-H}Z_n \Rightarrow N(0,\mu_p^2\mu^2_{r+1})$, $n\to\infty$
But $$n^{(r-1)H}R_n k_n= n^{1/2-H}n^{rH-1/2}R_n k_n\to 0,\ n\to\infty,$$
whence \eqref{roddH>1/2} follows. 

For $r=1$, $\sigma_{H,r}=0$,

The proof is now complete.
\end{proof}
Further we study the almost sure behavior of the mixed variations; for brevity, the phrase ``almost surely'' will be omitted.
\begin{proposition}\label{prop-asestim}
Let $\eps>0$ be arbitrary.

If $r=0$, then $S_n^{H,p,r} = o(n^{1/2+\eps})$, $n\to\infty$.

If $p$ and $r\ge 2$ are even, then 
\begin{itemize}
\item for $H\in(0,3/4]$ \  $S_n^{H,p,r} = o(n^{1/2+\eps})$, $n\to\infty$.
\item for $H\in(3/4,1)$ \  $S_n^{H,p,r} = o(n^{2H-1+\eps})$, $n\to\infty$.
\end{itemize}

If $p$ is odd, then for any $H\in(0,1)$ \ $S_n^{H,p,r} = o(n^{1/2+\eps})$, $n\to\infty$.

If $p$ is even and $r$ is odd, then
\begin{itemize}
\item for $H\in(0,1/2]$ \  $S_n^{H,p,r} = o(n^{1/2+\eps})$, $n\to\infty$.
\item for $H\in(1/2,1)$ \  $S_n^{H,p,r} = o(n^{H+\eps})$, $n\to\infty$.
\end{itemize}
\end{proposition}
\begin{proof}
We assume that $p$ and $r$ are even, $H\in (0,3/4)$, in other cases the argument is similar. 
Abbreviate $Q_n = n^{-1/2} S_n^{H,p,r}$. We need to show that $Q_n = o(n^\eps)$, $n\to\infty$. It is easy to check that $\ex{Q_n^2}\to \sigma^2_{H,r}\mu^2_p + \sigma_{p,r}^2$, $n\to\infty$. It follows that $\sup_{n\ge 1} \ex{Q_n^2} <\infty$. 

Clearly, $Q_n$ can be represented as a combination of multiple stochastic integrals with respect to some fixed Gaussian measure of order between $1$ and $p+r$. Then we can use the following well-known fact (see e.g.\ \cite[Corollary 7.36]{janson}): for any integer $p\ge l$, there exists a constant $C_l$ such that for all $n\ge 1$ \ $\ex{Q_n^{2l}}\le C_{l} \left(\ex{Q_n^2} \right)^l$.

Now take any integer $l \ge {\eps}^{-1}$ and write 
\begin{gather*}
\ex{\sum_{n= 1}^\infty \frac{Q_n^{2l}}{n^{2}}} = \sum_{n= 1}^\infty\ex{ \frac{Q_n^{l}}{n^{2}}}\le C_l\sum_{n= 1}^\infty \frac{\left(\ex{Q_n^2} \right)^l}{n^2}\\ \le C_l\left(\sup_{n\ge 1}\ex{Q_n^2} \right)^l\sum_{n= 1}^\infty \frac{1}{n^2}<\infty.
\end{gather*}
Therefore, the series $\sum_{n= 1}^\infty {Q_n^{2l}}/{n^{2}}$ converges almost surely; in particular, $Q_n = o(n^{1/l})$, $n\to\infty$, whence the statement follows.
\end{proof}
\section{Statistical estimation in mixed model}\label{statsec}

Now we turn to the question of parametric estimation in the mixed model
\begin{equation}\label{mixed model}
M_t^H =  a B_t^H + b W_t,\ t\in[0,T],
\end{equation}
where $a$, $b$ are non-zero numbers, which we assume to be positive, without loss of generality. Our primary goal is to construct a strongly consistent estimator for the Hurst parameter $H$, given a single observation of $M^H$.

It is well-known (see \cite{cheredito}) that for $H\in(3/4,1)$ the measure induced by $M^H$ in $C[0,T$] is equivalent to that of $b W$. Therefore, the property of almost sure convergence in this case is independent of $H$. Consequently, no strongly  consistent estimator for $H\in(3/4,1)$ based on a single observation of $M^H$ exists.
%
%
%For $H<1/2$ one can use standard quadratic variations to identify $H$, see below. So we will focus mainly on  $H\in(1/2,3/4)$, making necessary comments for other values of $H$.

In this section we denote  $\Delta_i^n X = X_{T(i+1)/n}-X_{Ti/n}$
and
$$
V^{H,p,r}_n = \sum_{i=0}^{n-1}  \lt(\Delta_i^n W\rt)^p \lt(\Delta_i^n B^H\rt)^r.
$$

\subsection{Statistical estimation based on quadratic variation}

Consider the quadratic variation of $M^H$, i.e.
\begin{gather*}
V^{H,2}_n:= \sum_{i=0}^{n-1}  \lt(\Delta_i^n M^H\rt)^2 = a^2 V^{H,0,2}_n + 2ab V^{H,1,1}_n + b^2 V^{H,2,0}_n.
\end{gather*}
Note that $V^{H,2}_n$ depends only on the observed process, so the notation

By the ergodic theorem, we have that $V^{H,0,2}_n\sim T^{2H} n^{1-2H}$, $V^{H,2,0}_n\to T$, $V^{H,1,1}_n = o(n^{1/2-H})$, $n\to\infty$. Therefore, the asymptotic behavior of $V_n^{H,2}$ depends on  whether $H<1/2$ or not. Precisely, for $H\in(0,1/2)$,
\begin{equation}\label{v2nh<1/2}
V^{H,2}_n\sim a^2 T^{2H} n^{1-2H}, n\to\infty,
\end{equation}
so the quadratic variation behaves similarly to that of a scaled fBm. 

For $H\in(1/2,1)$, 
\begin{equation}\label{v2nh>1/2}
V^{H,2}_n \to b^2 T, n\to\infty,
\end{equation}
so  the quadratic variation behaves similarly to that of a scaled Wiener process. 

Let us consider the cases $H<1/2$ and $H>1/2$ individually in more detail.

\subsubsection{$H\in(0,1/2)$}
%\paragraph{Consistent estimates for $H$}
We have seen above that this case is similar to the pure fBm case. Unsurprisingly, the same estimators work, which is precisely stated below.
\begin{theorem}\label{thm-h<1/2cons}
For $H\in(0,1/2)$, the following statistics
$$
\widehat{H}_{k} = \frac12\lt(1 - \frac1k \log_2 V_{2^k}^{H,2}\rt)
$$
and
$$
\widetilde{H}_{k} = \frac12\lt(\log_2 \frac{V_{2^k}^{H,2}}{V_{2^{k+1}}^{H,2}} + 1\rt)
$$
are strongly consistent estimators of the Hurst parameter $H$.
\end{theorem}
\begin{proof}
Write 
\begin{equation*}
\begin{gathered}
\log_2 V^{H,2}_{2^k} =  \log_2 \left(a^2 T^{2H} 2^{k(1-2H)}\right)+ \log_2 \left(1+ \frac{b^2}{a^2}T^{1-2H}2^{-k(1-2H)} + \zeta_k\right),
\end{gathered}
\end{equation*}
where 
\begin{equation*}
\zeta_k = \frac{a^2 \left(V_{2^k}^{H,0,2}-T^{2H}2^{k(1-2H)}\right)+ b^2\left(V_{2^k}^{H,2,0}-T\right)+ 2a b V_{2^k}^{H,1,1}} {a^2 T^{2H} 2^{k(1-2H)}}.
\end{equation*}
From Proposition~\ref{prop-asestim} it follows that for any $\eps>0$ \ $\zeta_k =  o(2^{k(-1/2+\eps)}) + o(2^{k(2H-3/2+\eps)})+ o(2^{k(H-1+\eps)})=o(2^{k(-1/2+\eps)})$, $k\to\infty$. 
Hence we have
\begin{equation}\label{log2vnh2}
\begin{gathered}
\log_2 V^{H,2}_{2^k} = 2 \log_2 a + 2H \log_2 T + (1-2H) k
\\ + O(2^{k(2H-1)})+ o(2^{k(-1/2+\eps)}), k\to\infty.
\end{gathered}
\end{equation}
In particular, 
$$
\log_2 V^{H,2}_{2^k} \sim 2 \log_2 a + 2H \log_2 T + (1-2H) k, k\to\infty,
$$
whence
 the result immediately follows.
\end{proof}
\begin{remark}
At the first sight, there is no clear advantage of $\widehat{H}_{k}$ or $\widetilde{H}_{k}$. But a careful analysis shows that $\widetilde{H}_{k}$ is better. Indeed, from \eqref{log2vnh2} it is easy to see that
\begin{equation}\label{hhatasym}
\widehat{H}_k = H - \frac{\log_2 a + H \log_2 T}{k} + o(k^{-1}),k\to\infty,
\end{equation}
while
\begin{equation} \label{htildeasym}
\widetilde{H}_k = H  + O(2^{k(2H-1)}) + o(2^{k(-1/2+\eps)}), k\to\infty.
\end{equation}
Now it is absolutely clear that $\widetilde{H}_k$ performs much better
(unless one hits the jackpot by having $a T^H=1$).
\end{remark}

Now we turn to the question of asymptotic normality of the estimators. Note that in the purely fractional case, the estimator $\widetilde{H}_k$ is asymptotically normal for all $H\in(0,1)$. In the mixed case, the analogy ends at $H=1/4$. 
\begin{proposition}\label{prop-tildeHasympnorm}
For $H\in(0,1/4)$,
$$
2^{k/2}\left( \widetilde{H}_k-H\right)\Rightarrow N(0,(\sigma'_{H})^2),\quad k\to\infty,
$$
where
\begin{gather*}
\sigma_H' = \frac{1}{2\log 2}\lt(\rho'_{H,0}+2\sum_{m=1}^\infty \rho'_{H,m}\rt)^{1/2},\\
\rho'_{H,m} = \mathsf{E}\biggl[\lt(\lt(B_1^H\rt)^2 - 2^{2H-1}\lt(B_{1/2}^H\rt)^2 - 2^{2H-1}\lt(B_1^H-B_{1/2}^H\rt)^2\rt)\\
\times \lt(\lt(B_{m+1}^H-B_{m}^H\rt)^2 - 2^{2H-1}\lt(B_{m+1/2}^H-B_{m}^H\rt)^2 - 2^{2H-1}\lt(B_{m+1}^H-B_{m+1/2}^H\rt)^2\rt)\biggr].
\end{gather*}
\end{proposition}
\begin{proof}
Write 
\begin{gather*}
\widetilde{H}_k-H = \frac12\lt(\log_2 \frac{V_{2^k}^{H,2}}{V_{2^{k+1}}^{H,2}}-(2H-1)\rt) = 
\frac1{2}\log_2 \frac{V_{2^k}^{H,2}}{2^{2H-1}V_{2^{k+1}}^{H,2}} \\
=  \frac1{2}\log_2 \left(\frac{V_{2^k}^{H,2}-2^{2H-1}V_{2^{k+1}}^{H,2}}{2^{2H-1}V_{2^{k+1}}^{H,2}}+1\right).
\end{gather*}
Since by \eqref{v2nh<1/2}
$$
\zeta_k:= \frac{V_{2^k}^{H,2}-2^{2H-1}V_{2^{k+1}}^{H,2}}{2^{2H-1}V_{2^{k+1}}^{H,2}}\to 0,\quad k\to\infty,
$$
we obtain
$$
\widetilde{H}_k-H =  
\zeta_k  \left(\frac 1{2\log 2}+o(1)\right),\quad k\to\infty. 
$$
Now write 
\begin{gather*}
V_{2^k}^{H,2}-2^{2H-1}V_{2^{k+1}}^{H,2}
= a^2 R_k^{H,0,2} + 2ab R_k^{H,1,1} + b^2 R_k^{H,2,0},
\end{gather*}
where
\begin{gather*}
R_k^{H,i,j} = V_{2^k}^{H,i,j} - 2^{2H-1}V_{2^{k+1}}^{H,i,j}, i,j\in\set{0,1,2}.
\end{gather*}
By Proposition~\ref{prop-asestim} we have for any $\eps\in(0,H)$ \
$V_{n}^{H,1,1} = o(n^{-H+\eps})$, $n\to\infty$, whence $R_k^{H,1,1} = o(2^{k(-H+\eps)}) = o(1)$, $k\to\infty$. Therefore, 
\begin{gather*}
\frac{2ab R_k^{H,1,1}}{2^{2H-1}V^{H,2}_{2^{k+1}}}\sim  \frac{2b R_k^{H,1,1}}{a^2 T^{2H} 2^{k(1-2H)}} = o(2^{k(2H-1)}),\quad k\to\infty.
\end{gather*}
By the ergodic theorem, $V_n^{H,2,0} \to  T$, $n\to\infty$, so 
\begin{gather*}
\frac{b^2 R_k^{H,2,0}}{2^{2H-1}V^{H,2}_{2^{k+1}}}\sim  \frac{b^2 R_k^{H,2,0}}{a^2 T^{2H} 2^{k(1-2H)}} = O(2^{k(2H-1)}),\quad k\to\infty.
\end{gather*}
Thus,  we get
\begin{equation}\label{tildeH-H}
\begin{gathered}
2^{k/2}\left( \widetilde{H}_k-H\right) =\left(\frac{a^2 2^{k/2} R_k^{H,0,2}}{2^{2H-1}V_{2^{k+1}}^{H,2}} + O(2^{k(2H-1/2)}) \right) \left(\frac{1}{2\log 2  }+o(1)\right)\\
= \frac{2^{k(2H-1/2)} R_k^{H,0,2}}{2 T^{2H}\log 2 }  + o(1),\quad k\to\infty.
\end{gathered}
\end{equation}Now write
$$
R_k^{H,0,2} = \sum_{m=0}^{2^{k}-1} \left(\left(\Delta_{2^k}^m B^H\right)^2 - 2^{2H-1}\left(\Delta_{2^{k+1}}^{2m} B^H\right)^2 - 2^{2H-1}\left(\Delta_{2^{k+1}}^{2m+1} B^H\right)^2\right).
$$
In view of the self-similarity of $B^H$, 
$$
R_k^{H,0,2} \overset{d}{=} 2^{-2Hk}T^{2H}\sum_{m=0}^{2^{k}-1} \eta_k,
$$
where 
$$
\eta_k = \lt(B_{k+1}^H-B_{k}^H\rt)^2 - 2^{2H-1}\lt(B_{k+1/2}^H-B_{k}^H\rt)^2 - 2^{2H-1}\lt(B_{k+1}^H-B_{k+1/2}^H\rt)^2.
$$
So we can apply CLT for stationary Gaussian sequence (see \cite{breuermajor})  and deduce that
\begin{gather*}
\frac{2^{k(2H-1/2)} R_k^{H,0,2}}{ T^{2H} } \overset{d}{=} 2^{-k/2}\sum_{m=0}^{2^k-1}\eta_k  \Rightarrow N(0,\sigma^2),\quad k\to\infty,
\end{gather*}
where 
$$
\sigma^2 = \ex{\eta_0^2} + 2\sum_{m=0}^\infty \ex{\eta_0\eta_m} = 
\rho'_{H,0} + 2\sum_{m=0}^\infty \rho'_{H,m}.
$$
Using this convergence and \eqref{tildeH-H}, we get the required statement with the help of  Slutsky's theorem.
\end{proof}

Now let $H\in(1/4,1/2)$. (We omit $H=1/4$ for two reasons: first, it is hard to distinguish this case statistically from $H\neq 1/4$; second, in this case it is shown exactly as in Proposition~\ref{prop-tildeHasympnorm} that $2^{k/2}(\widetilde{H}_k-H)$ converges to a non-central limit law.) 
In this case neither  $\widehat{H}_k$ nor $\widetilde{H}_k$ is asymptotically normal. In fact, a careful analysis of the proof of Proposition~\ref{prop-tildeHasympnorm} shows that $2^{k(1-2H)}(\widetilde{H}_k-H)$ converges to some constant. Nevertheless, it is possible to construct an asymptotically  normal estimator by cancelling this constant out. To this end, one has to consider 
$$
U_k^{H,2}= V_{2^{k}}^{H,2} - V_{2^{k+1}}^{H,2}
$$
instead of $V_{2^k}^{H,2}$. For well-definiteness we introduce the notation
$$
\log_{2+} x = \begin{cases}
\log_2 x,\ &x>0,\\
0, &x\le 0.
\end{cases}
$$

\begin{theorem}\label{prop-tildehk2h<1/2}
For $H\in(0,1/2)$, the statistic 
$$
\widetilde{H}_k^{(2)} = \frac12\lt(\log_{2+}\frac{U_{k}^{H,2}}{U_{k+1}^{H,2}}+1\rt)
$$
is a strongly consistent estimator of $H$, moreover, for any $\eps>0$, 
\begin{equation}\label{tildehk2}
\widetilde{H}_k^{(2)} = H + o(2^{k(-1/2+\eps)}),\ k\to\infty.
\end{equation}
The estimator $\widetilde{H}_k^{(2)}$ is asymptotically normal:
$$
2^{k/2}\left(\widetilde{H}_k^{(2)}-H\right)\Rightarrow N(0,(\sigma''_H)^2),\quad k\to\infty,
$$
with
\begin{gather*}
\sigma_H'' = \frac{1}{(2^{2-2H}-2)\log 2}\lt(\rho''_{H,0}+2\sum_{m=1}^\infty \rho''_{H,m}\rt)^{1/2},\\
\rho''_{H,m} = \mathsf{E}\biggl[\Big( s_{0,1}  - 
(c_H+1) \big(s_{0}^{1/2} + c_H(s_{1/2}^{1/2}\big) + c_H\big(s_{0}^{1/4} + s_{1/4}^{1/4} + s_{1/2}^{1/4} + s_{3/4}^{1/4}\big)\Big)\\
\times \Big( s_{m}^{1}  - 
(c_H+1) \big(s_{m}^{1/2} + s_{m+1/2}^{1/2}\big) + c_H \big(s_{m}^{1/4} + s_{m+1/4}^{1/4} + s_{m+1/2}^{1/4} + s_{m+3/4}^{1/4}\big)
\Big)\biggr];
\end{gather*}
here $s_{t}^h = \left(B_{t+h}^H- B_{t}^H\right)^2$, $c_H = 2^{2H-1}$.
\end{theorem}
\begin{proof}
The proof is similar to that of Proposition~\ref{prop-tildeHasympnorm}, so we will omit some details. Using the same transformations as there, we get
\begin{gather*}
\widetilde{H}^{(2)}_k-H = \frac{U_{k}^{H,2}-2^{2H-1}U_{k+1}^{H,2}}{2^{2H-1}U_{k+1}^{H,2}}  \left(\frac 1{2\log 2}+o(1)\right),\quad k\to\infty.
\end{gather*}
Expand 
\begin{gather*}
U_{k}^{H,2}-2^{2H-1}U_{k+1}^{H,2}
= a^2 P_k^{H,0,2} + 2ab P_k^{H,1,1} + b^2 P_k^{H,2,0},
\end{gather*}
where
\begin{gather*}
P_k^{H,i,j} = V_{2^k}^{H,i,j} - (c_H+1)V_{2^{k+1}}^{H,i,j} + c_H V_{2^{k+2}}^{H,i,j}, i,j\in\set{0,1,2}.
\end{gather*}
Similarly to $R_k^{H,1,1}$ in Proposition~\ref{prop-tildeHasympnorm}, for any $\eps>0$ \
 $P_k^{H,1,1} = o(2^{k(-H+\eps)})$, $k\to\infty$. 
Further, $P_{k}^{H,2,0}$ has a generalized chi-square distribution with $\ex{P_{k}^{H,2,0}}=0$ and 
$\ex{\left(P_{k}^{H,2,0}\right)^2} = O(2^{-k})$, $k\to\infty$. As in Proposition~\ref{prop-asestim}, we deduce that for any $\eps>0$ \ $P_k^{H,0,2} = o(2^{k(-1/2+\eps)})$, $k\to\infty$.

Further, from \eqref{v2nh<1/2} \ $U_k^{H,2}\sim a^2 T^{2H} (1-2^{1-2H})2^{k(1-2H)}$, $k\to\infty$. Combining the obtained asymptotics, we can write
\begin{gather*}
2^{k/2}\left( \widetilde{H}^{(2)}_k-H\right) =
\frac{2^{k(2H-1/2)} P_k^{H,0,2}}{2 T^{2H}(1-2^{1-2H})\log 2 }  + o(1),\quad k\to\infty,
\end{gather*}
whence we deduce the asymptotic normality exactly as in Proposition~\ref{prop-tildeHasympnorm}. 

The estimate \eqref{tildehk2} is obtained as in Proposition~\ref{prop-asestim}.
\end{proof}
\begin{remark}
Despite $\widetilde{H}_k^{(2)}$ has asymptotically a better rate of approximation that $\widetilde{H}_k$ for $H\in(1/4,1/2)$, we still do not recommend to use it, as the asymptotic variance is high; it is practically useless for $k\le 10$.  
\end{remark}

Now we turn to estimation of the scale coefficients $a$ and $b$. As it is known from \cite{vanzanten}, for $H\in(0,1/4)$ the measure induced by $M^H$ in $C[0,T]$ is equivalent to that of $a B^H$. This not only gives another explanation why the results for $H\in(0,1/4)$ are essentially the same as for fractional Brownian motion alone, but also has another important consequence: for $H\in(0,1/4)$ it is not possible to estimate $b$ consistently.
\begin{proposition}
For $H\in(0,1/2)$, the statistic
$$\widetilde{a}^2_k= 2^{k(2 \widetilde H_k - 1)}T^{-2\widetilde H_k} V_{2^k}^{H,2}$$
is a strongly consistent estimator of $a^2$.

For $H\in(1/4,1/2)$ the statistic
$$
\widetilde b^2_k = \frac{2^{1-2\widetilde{H}_k^{(2)}}V^{H,2}_{2^k}-V^{H,2}_{2^{k+1}}}{(2^{1-2\widetilde{H}^{(2)}_k}-1)T}
$$
is a strongly consistent estimator of $b^2$. 
\end{proposition}

\begin{proof}
First, observe that
$$
\frac{\widetilde{a}^2_k}{a^2} \sim 2^{2k(\widetilde H_k- H)} T^{2(H-\widetilde H_k)} \to 1,\quad k\to\infty,
$$
since $k(\widetilde H_k-H)\to 0$, $k\to\infty$,  by \eqref{htildeasym}. Hence we get the strong consistency of $\widetilde{a}^2_k$.

Concerning $\widetilde{b}^2_k$, define 
$$
\widehat b^2_k = \frac{2^{1-2H}V^{H,2}_{2^k}-V^{H,2}_{2^{k+1}}}{\big(2^{1-2H}-1\big)T}.
$$
It easily follows from \eqref{v2nh<1/2} that $\widehat b^2_k\to b^2$, $k\to\infty$. So it is enough to show that $\widetilde b^2_k - \widehat b^2_k\to 0$, $k\to\infty$. To this end, write
\begin{gather*}
\widetilde b^2_k - b^2 =\frac{\lt(2^{1-2\widetilde{H}^{(2)}_k}-
 2^{1-2H}\rt)
 V^{H,2}_{2^k}}{\big(2^{1-2\widetilde{H}^{(2)}_k}-1\big)T}
\\
{}+T^{-1}\lt(2^{1-2H}V^{H,2}_{2^k}-V^{H,2}_{2^{k+1}}\rt)
\lt(\big(2^{1-2\widetilde{H}^{(2)}_k}-1\big)^{-1}-\big(2^{1-2H}-1\big)^{-1}
\rt).
\end{gather*}
Obviously, the second term converges to zero. Due to \eqref{tildehk2}, for any $\eps>0$
\begin{gather*}
\big(2^{1-2\widetilde{H}^{(2)}_k}-2^{1-2H}\big)V^{H,2}_{2^k} \sim
-2^{2-2H}  (\widetilde{H}^{(2)}_k-H) a^2 T^{2H} 2^{k(1-2H)}\log 2 \\=2^{k(1-2H)}o(2^{k(-1/2+\eps)}),\ k\to\infty,
\end{gather*}
whence we deduce the strong consistency of $\widetilde{b}_k^2$ for $H\in(1/4,1/2)$, since $1-2H <1/2$.
\end{proof}

\subsubsection{$H\in(1/2,3/4)$}
Now we move to the case $H\in(1/2,1)$. In view of \eqref{v2nh>1/2}, both $\widehat{H}_{k}$ and $\widetilde{H}_{k}$ converge to $1/2$ for $H\in(1/2,1)$, so they are not suitable for estimating $H$.
The solution is to use $U_k^{H,2} = V_{2^k}^{H,2} - V_{2^{k+1}}^{H,2}$, rather than $V_{2^{k}}^{H,2}$, for the construction of estimators. The resulting estimators work also for $H\in(0,1/2)$.
\begin{theorem}
For $H\in(0,1/2)\cup(1/2,3/4)$, statistics
$$
\widehat{H}^{(2)}_{k} = \frac12\lt(1 - \frac1k \log_{2+} U_k^{H,2}\rt)
$$
and
$$
\widetilde{H}^{(2)}_{k} = \frac12\lt(\log_{2+} \frac{U_{k}^{H,2}}{U_{k+1}^{H,2}} + 1\rt)
$$
are strongly consistent estimators of the Hurst parameter $H$.
\end{theorem}
\begin{proof}
Write 
$$
U_k^{H,2} = a^2 Q_k^{H,0,2} + 2ab Q_k^{H,1,1} + b^2 Q_k^{H,2,0},
$$
where $Q_k^{H,i,j} = V_{2^k}^{H,i,j} - V_{2^{k+1}}^{H,i,j}$, $i,j\in{0,1,2}$. 
By the ergodic theorem, $Q_k^{H,0,2}\sim T^{2H} (1-2^{1-2H})2^{k(1-2H)}$, $k\to\infty$. By Proposition~\ref{prop-asestim}, for any $\eps>0$ \ $Q_{k}^{H,1,1} = o(2^{k(-H+\eps)})$, $k\to\infty$, and $$Q_{k}^{H,2,0} = 
\left(V_{2^k}^{H,2,0}-T\right) - \left(V_{2^{k+1}}^{H,2,0}-T\right) = o(2^{k(-1/2+\eps)}),\quad k\to\infty.
$$
Thus, we have 
\begin{equation}\label{ukh2asym}
U_k^{H,2} \sim a^2 T^{2H} (1-2^{1-2H})2^{k(1-2H)},\quad  k\to\infty,
\end{equation}
 which yields the proof.
\end{proof}
\begin{remark}
We will see in Section~\ref{simsec} that $\widetilde{H}^{(2)}_k$ performs very poorly, and $\widehat{H}^{(2)}_k$ performs somewhat better, despite having a worse asymptotic rate of convergence. 
\end{remark}
As in the case $H\in(0,1/2)$, the estimator $\widetilde{H}^{(2)}_k$ is asymptotically normal for $H\in(1/2,3/4)$; however, the limit Gaussian law comes out of the quadratic variation of the Wiener process, so the convergence rate is different, and the expression for the asymptotic variance is much simpler.
\begin{theorem}
For $H\in(1/2,3/4)$ and any $\eps>0$, the estimator
$$
\widetilde{H}_k^{(2)} = \frac12\lt(\log_{2+}\frac{U_{k}^{H,2}}{U_{k+1}^{H,2}}+1\rt)
$$
satisfies
\begin{equation}\label{tildehk2h>1/2}
\widetilde{H}_k^{(2)} = H + o(2^{k(2H-3/2+\eps)}),\ k\to\infty.
\end{equation}
It is asymptotically normal:
$$
2^{k(3/2-2H)}\left(\widetilde{H}_k^{(2)}-H\right)\Rightarrow N(0,(\sigma''_H)^2),\quad k\to\infty,
$$
with
\begin{gather*}
(\sigma_H'')^2 = \frac{(2^{4H-1}+1)T^{1-2H}}{(2-2^{2-2H})\log 2}.
\end{gather*}
\end{theorem}
\begin{proof}
As in the proof of Proposition~\ref{prop-tildehk2h<1/2}, write
\begin{gather*}
\widetilde{H}^{(2)}_k-H = \frac{U_{k}^{H,2}-2^{2H-1}U_{k+1}^{H,2}}{2^{2H-1}U_{k+1}^{H,2}}  \left(\frac 1{2\log 2}+o(1)\right),\quad k\to\infty,
\end{gather*}
and expand
\begin{gather*}
U_{k}^{H,2}-2^{2H-1}U_{k+1}^{H,2}
= a^2 P_k^{H,0,2} + 2ab P_k^{H,1,1} + b^2 P_k^{H,2,0},
\end{gather*}
where
\begin{gather*}
P_k^{H,i,j} = V_{2^k}^{H,i,j} - (c_H+1)V_{2^{k+1}}^{H,i,j} + c_H V_{2^{k+2}}^{H,i,j}, i,j\in\set{0,1,2},
\end{gather*}
and $c_H = 2^{2H-1}$.

We have from the proof of Theorem~\ref{prop-tildehk2h<1/2} that for any $\eps>0$ 
 $P_k^{H,1,1} = o(2^{k(-H+\eps)})$, 
$P_{k}^{H,0,2} = o(2^{k(-H+\eps)})$, $k\to\infty$.
Therefore, using \eqref{ukh2asym}, we get
\begin{gather*}
2^{k(3/2-2H)}\left(\widetilde{H}_k^{(2)}-H\right) =
\frac{2^{k/2} P_k^{H,2,0}}{2 T^{2H}(1-2^{1-2H})\log 2 }  + o(1),\quad k\to\infty.
\end{gather*}
We can write $P_k^{H,2,0}=\sum_{m=0}^{2^{k}-1}\kappa_{k,m}$, where
\begin{gather*}
\kappa_{k,m} = \bigg(\left(\Delta_{2^k}^m W\right)^2 - (c_H+1) \left(\left(\Delta_{2^{k+1}}^{2m} W\right)^2+\left(\Delta_{2^{k+1}}^{2m+1} W\right)^2\right)
\\+ c_H \left(\left(\Delta_{2^{k+2}}^{4m} W\right)^2+\left(\Delta_{2^{k+2}}^{4m+1} W\right)^2+\left(\Delta_{2^{k+2}}^{4m+2} W\right)^2+\left(\Delta_{2^{k+2}}^{4m+3} W\right)^2\right)\bigg).
\end{gather*}
The random variables $\set{\kappa_{k,m}, m=0,\dots,2^k-1}$ are iid with $\ex{\kappa_{k,m}}=0$ and $\ex{\kappa_{k,m}^2} = T 2^{-2k}(2^{4H-1}+1)$. Therefore, by the classical CLT, 
$$
2^{k/2} P_k^{H,2,0} \Rightarrow N(0,T (2^{4H-1}+1)),\quad k\to\infty,
$$
whence we get by Slutsky's theorem,
$$
2^{k(3/2-2H)}\left(\widetilde{H}_k^{(2)}-H\right) \Rightarrow N(0,(\sigma_H'')^2),\quad k\to\infty.
$$
Again, the estimate \eqref{tildehk2h>1/2} is obtained as in Proposition~\ref{prop-asestim}.
\end{proof}
The estimation of the scale coefficient $a$ is similar to the case $H\in(0,1/2)$, but we have to use $U_k^{H,2}$ and $\widetilde H^{(2)}_k$ instead of $V_{2^k}^{H,2}$ and $\widetilde H_k$; the resulting estimator works also for $H\in(0,1/2)$. Estimating $b^2$ is a lot easier, thanks to \eqref{v2nh>1/2}.
\begin{proposition}
For $H\in(0,1/2)\cup(1/2,3/4)$, the statistic
$$\hat{a}^2_k= 2^{k(2 \widetilde H^{(2)}_k - 1)}T^{-2\widetilde H^{(2)}_k}(1-2^{1-2\widetilde H^{(2)}_k})^{-1} U_{k}^{H,2}$$
is a strongly consistent estimator of $a^2$.

For $H\in(1/2,1)$, the statistic
$$
\hat b^2_k = \frac{V^{H,2}_{2^k}}{T}
$$
is a strongly consistent estimator of $b^2$. 
\end{proposition}
\begin{proof}
In view of \eqref{ukh2asym}, 
$$
\frac{\hat{a}^2_k}{a^2} \sim 2^{2k(\widetilde H_k^{(2)}- H)}T^{2(H-\widetilde H^{(2)}_k)}\frac{1-2^{1-2H}}{1-2^{1-2\widetilde{H}_k^{(2)}}}  \to 1,\ k\to\infty,
$$
since $k(\widetilde H^{(2)}_k-H)\to 0$, $k\to\infty$,  by \eqref{tildehk2h>1/2}. Hence we get the strong consistency of $\hat{a}^2_k$. The strong consistency of $\hat b^2_k$ is obvious from \eqref{v2nh>1/2}.
\end{proof}

\subsubsection{$H\in (3/4,1)$}\label{(3/4,1)sec}

As we have already mentioned in the beginning of this section, it is impossible to make conclusions about the value of $H$ in this case. In fact, we have
\begin{gather*}
n^{1/2}(V_n^{H,2}-b^2 T) \Rightarrow b^2  T\,  N(0,2),\quad n\to\infty.
\end{gather*}
Indeed, $ n^{1/2}(V_n^{H,2,0}- T)\Rightarrow N(0,2)$, $n\to\infty$, by the classical CLT; $V_n^{H,0,2} \sim T^{2H}n^{1-2H} = o(n^{-1/2})$, $n\to\infty$, and for any $\eps>0$ \ $V_n^{H,1,1} = o(n^{-H+\eps})$, $n\to\infty$, due to Proposition~\ref{prop-asestim}. This means that the behaviour of $V_n^{H,2}$ is essentially the same as that of the quadratic power variation  of Wiener process, in particular, so it says nothing about $H$. 

Nevertheless, we will study the behaviour of quadratic variation in more detail in order to be able to distinguish between the cases $H<3/4$ and $H>3/4$ statistically.

Define 
$$
Z_k = \frac{2^{k/2}}{b^2T } U_k^{H,2}.
$$
\begin{proposition}\label{H=34-Z_n}
For $H\in(3/4,1)$, the sequence $(Z_k,Z_{k+1},\dots)$ converges
in distribution as $k\to\infty$ to a sequence
$(\zeta_1,\zeta_2,\dots)$ of independent standard Gaussian variables.
\end{proposition}

\begin{remark}
We emphasize a sharp contrast with the case $H\in(1/2,3/4)$, where the sequence $\set{Z_k,k\ge 1}$ has a positive limit in view of \eqref{ukh2asym}, hence, it eventually becomes positive. This clearly gives a possibility to distinguish statistically between cases $H\in(1/2,3/4)$ and $H\in(3/4,1)$. (See \ref{simuh2k1/2-3/4} for comparative simulations.)
\end{remark}
\begin{proof}
Define
\begin{gather*}
\xi_k = \frac{2^{k/2}}{\sqrt 2 T}\lt(V_{2^k}^{H,2,0} -  T\rt)
= \frac{2^{-k/2}}
{\sqrt 2 T}\sum_{i=0}^{2^k-1}\lt(2^{k}\lt(\Delta_i^{2^k} W\rt)^2- T \rt).
\end{gather*}
By the classical CLT, $\xi_k\Rightarrow N(0,1)$, $k\to\infty$, so we need to study the collective behaviour.
To this end, observe that the vector $(\xi_k,\xi_{k+1},\dots,\xi_{k+m})$ can be represented as a sum of independent vectors
$$
(\xi_k,\xi_{k+1},\dots,\xi_{k+m}) = \sum_{i=0}^{2^{k}-1} \zeta_{k,i},
$$
where the $j$th coordinate of $\zeta_{k,i}$, $j=0,1,2,\dots,m$,
is
$$
\zeta_{k,i,j} = \frac{2^{-k/2}}{\sqrt 2 T}\sum_{l=0}^{2^j-1}
\lt(2^{2(k+j)}\lt(\Delta_{l + i 2^{j}}^{2^{k+j}} W\rt)^2- T
\rt).
$$
(We simply group terms on the intervals of the partition $\set{i T2^{-k}, i=0,\dots,2^{k} }$.)
Therefore, we can apply a vector CLT and deduce that for every $m\ge 0$ the vector
$(\xi_k,\xi_{k+1},\dots,\xi_{k+m})$ converges in distribution to an $(m+1)$-dimensional centered Gaussian vector as $k\to\infty$.
Consequently, the sequence $(\xi_k,\xi_{k+1},\xi_{k+2},\dots)$ converges to a centered stationary Gaussian sequence as $k\to\infty$. 

We have seen above that  $V^{H,2}_n = b^2 V^{H,2,0}_n + o(n^{-1/2})$, $n\to\infty$. Therefore, $Z_k = b^2 \left(\xi_k-\sqrt{2}\xi_{k+1}\right) + o(1)$, $k\to\infty$, so by Slutsky's theorem the sequence $(Z_k,Z_{k+1},Z_{k+2},\dots)$ also converges to a centered stationary Gaussian sequence. It is straightforward to check that the limit covariance is that of the i.i.d.\ standard Gaussian sequence, whence the result follows. 
\end{proof}
\begin{remark}
For $H=3/4$, an analogue of Proposition~\ref{H=34-Z_n} can be proved,
that is, $(Z_k,Z_{k+1},\dots)$ converges
in distribution as $k\to\infty$ to a sequence
$(\zeta_1,\zeta_2,\dots)$ of independent Gaussian variables with unit variance. 
However, it can be checked that
the limiting stationary distribution now has a positive mean, namely, $\ex{\zeta_1}={a^2}{b^{-2}}T^{1/2}(1-2^{-1/2})$. As long as this value depends on how big is $a$ compared to $b$, we might be unable to distinguish this case from $H>3/4$. On the other hand, if  $b$ is small relative to $a$, it might be hard to distinguish this case from $H<3/4$.
\end{remark}

\subsection{Statistical estimation using 4th power variation}
It was mentioned in the previous section that the performance of quadratic variation estimators in the case $H\in(1/2,3/4)$ is not very satisfactory. One could try to improve it by considering quartic variation of $M^H$
\begin{gather*}
V^{H,4}_n:=\sum_{k=0}^{n-1} \lt(\Delta_k^n M^H\rt)^4
%= \sum_{i=0}^{4} \binom4i a^i b^{4-i} \sum_{k=0}^{n-1} \lt(\Delta_k^n B^H\rt)^i \lt(\Delta_k^n W\rt)^{4-i}=:\sum_{i=0}^{4} \binom4i a^i b^{4-i} P_n^i.
= \sum_{i=0}^{4} \binom4i a^i b^{4-i} V^{H,4-i,i}_n.
\end{gather*}
As for the quadratic variation, we have to cancel out the leading term, considering 
$$
U_k^{H,4} = V_{2^k}^{H,4} - 2V_{2^{k+1}}^{H,4}.
$$
\begin{theorem}\label{thm-quartic}
The statistics
$$
\widehat{H}^{(4)}_{k} = -  \frac1{2k} \log_{2+} U_k^{H,4}
$$
and
$$
\widetilde{H}^{(4)}_{k} = \frac12\log_{2+} \frac{U_{k}^{H,4}}{U_{k+1}^{H,4}} 
$$
are strongly consistent estimators of the Hurst parameter $H\in(1/2,3/4)$
in the mixed model \eqref{mixed model}.
\end{theorem}
\begin{proof}
By the ergodic theorem,  $V_n^{H,2,2} \sim T^{2H+1} n^{-2H}$,
$V_n^{H,0,4} \sim 3 T^{4H}n^{1-4H}$, $n\to\infty$. Further, from Proposition~\ref{prop-asestim} for any $\eps>0$ \ $V_n^{H,4,0}- Tn^{-1} = o(n^{-3/2+\eps})$, $V_n^{H,3,1} = o(n^{-1-H+\eps})$, $V_n^{H,1,3} = o(n^{-3H+\eps})$, $n\to\infty$. 

Now write
$$
U_k^{H,4}= \sum_{i=0}^{4} \binom4i a^i b^{4-i} U^{H,4-i,i}_k,
$$
where $U^{H,4-i,i}_k = V^{H,4-i,i}_{2^k} -V^{H,4-i,i}_{2^{k+1}}$, $i=0,\dots,4$. We have $U_k^{H,2,2} \sim T^{2H+1} (1-2^{1-2H})2^{-2Hk}$, $U_k^{H,0,4} = O(2^{k(1-4H)})= o(2^{-2Hk})$, $k\to\infty$, and for any $\eps>0$ 
\begin{gather*}
U_k^{H,4,0} = \left(V_{2^k}^{H,4,0}- T2^{-k}\right) -  2\left(V_{2^{k+1}}^{H,4,0}- T2^{-k-1}\right) = o(2^{k(-3/2+\eps)}),\\
U_k^{H,3,1} = o(2^{k(-1-H+\eps)}),\ U_k^{H,1,3} = o(2^{k(-3H+\eps)}),\quad k\to\infty.
\end{gather*}
Collecting all the terms, we get
\begin{equation*}
U_k^{H,4} \sim 4T^{2H+1} (1-2^{1-2H})2^{-2Hk}, \quad k\to\infty.
\end{equation*}
Hence, the assertion follows.
\end{proof}
\begin{remark}
Both these estimators are quite poor. A regression of several values of $\log_2 U_k^{H,4}$ on $k$ leads to a much better estimator. However, as numerical experiments in Section~\ref{simsec} suggest, it is better to use the quadratic variation based estimators (which are  not very efficient as well). 
\end{remark}

\subsection{Estimation of Hurst parameter for known scale coefficients}

When the scale coefficients $a$ and $b$ are known, the estimation procedure significantly simplifies, and the quality of estimators is improved. It may seem unnatural at a first glance that the scale coefficients are known while $H$ is not. However, the case where $b$ is known is quite natural, as we can have known white noise amplitude with unknown long-range perturbation of this white noise. The cases of known $a$ or known both coefficients are less natural, but there is no reason to omit this cases considering only the case of known $b$.

\begin{theorem}
If $a$ is known, then the statistic
$$
\widehat{H}_k(a) = \frac{k + 2\log_2 a -\log_2 V_{2^k}^{H,2}}{2(k -\log_2 T)}
$$
is a strongly consistent estimator of $H\in(0,1/2)$, moreover, for any $\eps>0$,
$$
\widehat{H}_k(a) = H + O(2^{k(2H-1)}) + o(2^{k(-1/2+\eps)}),\quad k\to\infty.
$$

If $b$ is known, then the statistic 
$$
\widetilde{H}_k(b)=\frac{1}{2}\left(\log_{2+}\frac{V_{2^k}^{H,2}-b^2T}{V_{2^{k+1}}^{H,2}-b^2T}+1\right)
$$
is a consistent estimator of $H\in(0,3/4)$,  moreover, for any $\eps>0$,
$$
\widetilde{H}_k(b) = H + o(2^{k(-1/2+\eps)})+ o(2^{k(2H-3/2+\eps)}),\quad k\to\infty.
$$

If $a$ and $b$ are known, then the statistic
$$
\widehat{H}_k(a,b) = \frac{k + 2\log_2 a -\log_{2+} (V_{2^k}^{H,2}-b^2 T)}{2(k -\log_2 T)}
$$
is a strongly consistent estimator of $H\in(0,3/4)$, moreover, for any $\eps>0$,
$$
\widehat{H}_k(a,b) = H  + o(2^{k(-1/2+\eps)})+o(2^{k(2H-3/2+\eps)}),\quad k\to\infty.
$$
\end{theorem}
\begin{proof}
The statement for $\widehat{H}_k(a)$ follows immediately from \eqref{log2vnh2}. To prove the statement for $\widetilde{H}_k(b)$ and $\widehat{H}_k(a,b)$, note that, in view of \eqref{v2nh<1/2}, $V^{H,2}_{2^k}>b^T$ for sufficiently large $k$. 
Therefore, we can write, as in the proof of Theorem~\ref{thm-h<1/2cons}, 
\begin{equation*}
\begin{gathered}
\log_{2}(V^{H,2}_{2^k}-b^2T) = \log_2 \left(a^2 T^{2H} 2^{k(1-2H)}\right)+ \log_2 \left(1 + \zeta_k\right),
\end{gathered}
\end{equation*}
with the same $\zeta_k$; in particular, for $H\in(0,1/2]$ and any $\eps>0$, $\zeta_k = o(2^{k(-1/2+\eps)})$, $k\to\infty$. For $H\in(1/2,3/4)$, $\zeta_k = o(2^{k(-H+\eps)})+ o(2^{k(2H-3/2+\eps)})+ o(2^{k(H-1+\eps)}) = o(2^{k(2H-3/2+\eps)})$, $k\to\infty$. This implies the statement for both $\widetilde{H}_k(b)$ and $\widehat{H}_k(a,b)$.
\end{proof}
\begin{remark}
It can be shown that $\widehat H_k(a)$ is asymptotically normal for $H\in(0,1/4)$, $\widetilde H_k(b)$, for $H\in(1/2,3/4)$, $\widehat H_k(a,b)$, for $H\in(0,3/4)$. This is not our main concern here, so we skip the asymptotic normality results.
\end{remark}

\section{Simulations}\label{simsec}
In each procedure we take $T=3$, $a=b=1$, $n=2^{20}$ and use the circulant method to simulate values of $B^H$ on the uniform partition $\{iT/n,i=0,1,\dots,n \}$ of $[0,T]$. 
For each value of the Hurst parameter, we simulate $10$ trajectories of the fBm. Then for each estimator $\check H$  we compute the average $\mu \check H$ of the $10$ obtained values and  the average error ${\delta} \check H$, i.e.\ the average of the values $|\check H-H|$. Where possible, we make similar procedure for $a$ and/or $b$.

Each simulation takes about 6~seconds on Intel Core i5-3210M processor, computing all estimators takes about 20~milliseconds.  

\subsection{Estimators based on quadratic variation}

\subsubsection{$H\in(0,1/2)$}

\begin{table}[htbp]
\caption{Values of the quadratic variation based estimators for $H\in(0,1/2)$}
\begin{center}
\begin{tabular} {cccccccccc}%{|c|c|c|c|c|c|c|c|c|c|}
\hline
$H$ & $0.05$ & $0.1$ & $0.15$ & $0.2$ & $0.25$ & $0.3$ & $0.35$ & $0.4$ & $0.45$ \\\hline
$\mu \widehat{H}_{20}$ &$.0460$ & $.0921$ & $.1381$ & $.1841$ & $.2301$ & $.2760$ & $.3215$ & $.3656$ & $.4055$ \\\hline
$\delta \widehat{H}_{20}$ &$.0040$ & $.0079$ & $.0119$ & $.0159$ & $.0199$ & $.0240$ & $.0285$ & $.0344$ & $.0445$ \\\hline
$\mu \widetilde{H}_{19}$ & $.0497$ & $.1001$ & $.1505$ & $.1997$ & $.2502$ & $.3016$ & $.3538$ & $.4082$ & $.4609$ \\\hline
$\delta \widetilde{H}_{19}$ & $.0009$ & $.0013$ & $.0010$ & $.0012$ & $.0009$ & $.0016$ & $.0038$ & $.0082$ & $.0109$ \\\hline
$\mu\widetilde{H}^{(2)}_{18}$ & $.0486$ & $.1012$ & $.1507$ & $.1978$ & $.2481$ & $.3011$ & $.3521$ & $.3979$ & $.4396$ \\\hline
$\delta\widetilde{H}^{(2)}_{18}$ & $.0030$ & $.0035$ & $.0050$ & $.0062$ & $.0046$ & $.0076$ & $.0082$ & $.0117$ & $.0248$ \\\hline
$\mu \widetilde{a}_{19}$ & $.9979$ & $1.019$ & $1.007$ & $.9972$ & $1.003$ & $1.0231$ & $1.061$ & $1.153$ & $1.299$ \\\hline
$\delta \widetilde{a}_{19}$ & $.0190$ & $.0161$ & $.0131$ & $.0154$ & $.0119$ & $.0231$ & $.0611$ & $.1532$ & $0.2994$ \\\hline
%$\mu\widetilde{b}_{18}$ & &  &  &  &  & $1.4375$ & $.7788$ & $.9662$ & $1.062$ \\\hline
%$\delta{b}_{18}$ & &  &  &  &  & $1.0375$ & $.4566$ & $.5013$ & $.8670$ \\\hline
\end{tabular}
\end{center}
\end{table}

In table I, we compare the estimators $\widehat{H}_{20}$, $\widetilde{H}_{19}$, $\widetilde{H}_{18}^{(2)}$ (observe that all these estimators are based on the values of fBm on the chosen partition). We also give  values of the estimator $\widetilde{a}_{19}$; the estimator $\widetilde{b}_{18}$ is quite bad: 2--5 values of $\widetilde{b}_{18}^2$ out of 10 are negative, others are quite away from the true value, so we do not give its values.

The results show that the estimator $\widetilde{H}_{19}$ has consistently the best performance. For $H>1/4$, a positive bias is visible, which is not surprising as it can be checked using the same transformations as in Proposition~\ref{htildeasym} that in this case $$\widetilde{H}_k - H \sim (1-2^{2H-1}) a^{-2} b^2 T^{1-2H}2^{k(2H-1)},\quad k\to\infty.$$

The estimator $\widehat{H}_{19}$ underestimates all values of $H$ by around $8$~\%. The underestimation follows from \eqref{hhatasym}, since $a T^H>1$. Finally, the relative error of $\widetilde{H}^{(2)}_{18}$ is larger than that of $\widetilde H_{19}$.

The estimator $\widetilde{a}_{19}$ is quite reliable, especially for smaller values of $H$; for $H>1/4$ it has a positive bias (inherited from $\widetilde{H}_k$). 

\subsubsection{$H\in(1/2,3/4)$}
Table II compares estimators $\widehat{H}_{19}^{(2)}$ and $\widetilde{H}_{18}^{(2)}$ of Hurst parameter $H$. It also contains a ``regression'' estimator $\bar{H}^{(2)}$ obtained in the following way: we consider the linear regression of $\{\log_{2+} U_j^{H,2},j=m,m+1,\dots,19\}$ on $\{m,m+1,\dots,19\}$, where $m=11,12,\dots,15$, and take the best regression (in terms of the coefficient of determination). If $\bar{r}^{(2)}$ is the coefficient of the best linear regression, we set $\bar{H}^{(2)} = (1-\bar{r}^{(2)})/2$. We also give the estimator $\hat{b}_{20}$. Due to uselessness of the estimator $\hat{a}_{20}$, we do not present its values.

It is clear that none of the estimators is reliable: average errors are in most cases comparable to the length of the range $(1/2,3/4)$, so they are quite useless. Only the performance of  $\widehat{H}^{(2)}_{19}$ in the range $0.575$--$0.7$ is acceptable, but one should be aware of a positive bias.

It is interesting to note that the errors of both $\widetilde{H}^{(2)}_{18}$ and $\bar{H}^{(2)}$ explode for $H>5/8$. We admit that we found no explanation for this phenomenon.

\begin{table}[htbp]
\caption{Values of the quadratic variation based estimators for $H\in(1/2,3/4)$}
\begin{center}
\begin{tabular}{cccccccccc}
\hline
$H$ & $0.525$ & $0.55$ & $0.575$ & $0.6$ & $0.625$ & $0.65$ & $0.675$ & $0.7$ & $0.725$ \\\hline
$\mu\widehat{H}^{(2)}_{19}$ &$.6095$ & $.6082$ & $.6135$ & $.6266$ & $.6432$ & $.6645$ & $.6850$ & $.7065$ & $.6582$ \\\hline
$\delta\widehat{H}^{(2)}_{19}$ &$.0845$ & $.0582$ & $.0385$ & $.0266$ & $.0182$ & $.0155$ & $.0174$ & $.0266$ & $.0711$ \\\hline
$\mu\widetilde{H}^{(2)}_{18}$ & $.4774$ & $.5172$ & $.5343$ & $.5877$ & $.6499$ & $.5925$ & $.7888$ & $.6971$ & $.3320$ \\\hline
$\delta\widetilde{H}^{(2)}_{18}$ & $.1383$ & $.1010$ & $.1206$ & $.0822$ & $.1209$ & $.2231$ & $.4801$ & $.5124$ & $.6144$ \\\hline
$\mu\bar{H}^{(2)}$ & $.6002$ & $.6002$ & $.5811$ & $.5959$ & $.6390$ & $.7819$ & $.8079$ & $.8297$ & $.5817$ \\\hline
$\delta\bar{H}^{(2)}$ & $.0965$ & $.0663$ & $.0441$ & $.0392$ & $.0392$ & $.1319$ & $.1528$ & $.1931$ & $.4009$ \\\hline
$\mu\hat{b}_{20}$ & $1.236$ & $1.131$ & $1.071$ & $1.038$  & $1.021$ & $1.011$ & $1.006$ & $1.003$ & $1.002$ \\\hline
$\delta\hat{b}_{20}$ & $.2359$ & $.1312$ & $.0710$ & $.0377$  & $.0206$ & $.0108$ & $.0056$ & $.0034$ & $.0017$ \\\hline
%$\mu\hat{a}_{18}$ & $.8398$ & $.7864$ & $1.174$ & $1.775$  & $2.536$ & $5.973$ & $\infty$ & $\infty$ & $\infty$ \\\hline
%$\delta\hat{a}_{18}$ & $1.2631$ & $1.081$ & $1.455$ & $1.203$  & $2.073$ & $5.987$ & $\infty$ & $\infty$ & $\infty$ \\\hline
\end{tabular}
\end{center}
\end{table}

\subsubsection{$H\in(3/4,1)$ versus $H\in(1/2,3/4)$}\label{simuh2k1/2-3/4}
Table III contains  values of  $\{[10^4 U_{k}^{H,2}],k=10,11,12,\dots,19\}$ 
for $H=0.7$ and $H=0.8$. The difference is clearly visible: for $H=0.7$ 
the sequence is positive, while for $H=0.8$ there is a plenty of negative 
values. 
\begin{table}[htbp]
\caption{Scaled values of  $U_k^{H,2}$ 
for $H=0.7$ and $H=0.8$}
\begin{center} % \setlength{\tabcolsep}{3pt}
\begin{tabular}{ccccccccccc}
\hline
$H=0.7$ & $869$ & $649$ & $523$ & $3$ & $260$ & $18$
& $78$  &$98$ & $53$ & $50$
\\
\hline
$H=0.8$ & $665$ & $-620$ &  $482$ & $-475$    & $8$ & $-29$ & $-104$ & $-71$  &  $-78$ & $-28$\\\hline
\end{tabular}
\end{center}
\end{table}

\subsection{Estimators based on quartic variation}

\subsubsection{$H\in(1/2,3/4)$}
Table IV contains estimators $\widehat{H}_{20}^{(4)}$ and $\widetilde{H}_{20}^{(4)}$ of Hurst parameter $H$, the values of $H$ range from $0.525$ to $0.725$ with step $0.025$. We also give a ``regression'' estimator $\bar{H}^{(4)}$. It is obtained in the following way: we consider the linear regression of $\{\log_{2+} U_j^{H,4},j=m,m+1,\dots,19\}$ on $\{m,m+1,\dots,20\}$, where $m=11,12,\dots,16$, and take the best regression (in terms of the coefficient of determination). If $\bar{r}^{(4)}$ is the coefficient of the best linear regression, we set $\bar{H}^{(4)} = -\bar{r}^{(4)}/2$. 

\begin{table}[htbp]
\caption{Values of the quartic variation based estimators for $H\in(1/2,3/4)$}
\begin{center}
\begin{tabular}{cccccccccc}
	\hline
	              $H$                & $0.525$ & $0.55$  & $0.575$ &  $0.6$  & $0.625$ & $0.65$  & $0.675$ &  $0.7$  & $0.725$ \\ \hline
	  $\mu\widehat{H}^{(4)}_{19}$    & $.4839$ & $.4876$ & $.4990$ & $.5138$ & $.5305$ & $.5523$ & $.5164$ & $.3688$ & $.4316$ \\ \hline
	 $\delta\widehat{H}^{(4)}_{19}$  & $.0411$ & $.0624$ & $.0760$ & $.0862$ & $.0945$ & $.0977$ & $.1568$ & $.3312$ & $.2934$ \\ \hline
	 $\mu\widetilde{H}^{(4)}_{18}$   & $.5345$ & $.4951$ & $.5103$ & $.6367$ & $.6301$ & $.7068$ & $.3463$ & $.7063$ & $.6531$ \\ \hline
	$\delta\widetilde{H}^{(4)}_{18}$ & $.0884$ & $.1369$ & $.1303$ & $.1232$ & $.2051$ & $.3834$ & $.5466$ & $.7973$ & $.6889$ \\ \hline
	       $\mu\bar{H}^{(4)}$        & $.5676$ & $.5767$ & $.6257$ & $.6129$ & $.6718$ & $.9245$ & $1.2051$ & $.0410$ & $1.128$ \\ \hline
	     $\delta\bar{H}^{(4)}$       & $.0637$ & $.0492$ & $.0507$ & $.0661$ & $.1149$ & $.3307$ & $.8691$ & $1.277$ & $1.667$ \\ \hline
\end{tabular}
\end{center}
\end{table}
We see that the estimators based on the quartic variation are quite useless and definitely worse than those based on the quadratic variation. Again, the errors of $\widetilde{H}^{(4)}_{18}$ and $\mu\bar{H}^{(4)}$ explode for $H\ge 5/8$. In contrast to the quadratic variation case, now this phenomenon can be easily explained. The fact is that the nature of the error changes at $H=5/8$: for $H<5/8$, the error comes from the term $U_k^{H,0,4}$ (in the notation of the proof of Theorem~\ref{thm-quartic}), which behaves quite smoothly, but for $H\ge 5/8$, the main contribution comes from the fluctuations of $U_k^{H,4,0}$, which are much wilder.

\subsection{Estimation when $a$ and $b$ are known}
Table V gives estimators $\widehat{H}_{20}(a)$ and $\widehat{H}_{20}(a,b)$ for $H$ from $0.05$ to $0.45$ with the step $0.05$. Since the errors are very small, we multiply them by $100$.
We can see that the estimator $\widehat{H}_{20}(a)$ is comparable to $\widetilde{H}_{20}(a,b)$ for $H\le 1/4$; then it becomes worse, but it uses only knowledge of $a$.
\begin{table}[htbp]
\caption{Values of the  estimators for $H\in(0,1/2)$ and known scale coefficients}
\begin{center} 
\begin{tabular} {cccccccccc}
\hline
$H$ & $0.05$ & $0.1$ & $0.15$ & $0.2$ & $0.25$ & $0.3$ & $0.35$ & $0.4$ & $0.45$ \\\hline
$\mu \widehat{H}_{20}(a)$ &$.05$ & $.1$ & $.15$ & $.2$ & $.249$ & $.299$ & $.349$ & $.397$ & $.439$ \\\hline
$100\delta \widehat{H}_{20}(a)$ &$.012$ & $.007$ & $.008$ & $.007$ & $.009$ & $.034$ & $.113$ & $.352$ & $1.08$ \\\hline
%$\mu\widetilde{H}_{19}(b)$ & $.05$ & $.101$ & $.149$ & $.201$ & $.252$ & $.306$ & $.37$ & $.461$ & $.6384$ \\\hline
%100$\delta\widetilde{H}_{19}(b)$ & $.108$ & $.15$ & $.108$ & $.118$ & $.187$ & $.583$ & $1.96$ & $6.13$ & $18.84$ \\\hline
$\mu \widehat{H}_{20}(a,b)$ & $.05$ & $.1$ & $.15$ & $.2$ & $.25$ & $.3$ & $.35$ & $.4$ & $.45$ \\\hline
$100\delta \widehat{H}_{20}(a,b)$ & $.006$ & $.004$ & $.005$ & $.006$ & $.003$ & $.003$ & $.003$ & $.006$ & $.006$ \\\hline
\end{tabular}
\end{center}
\end{table}

Table IV contains estimators $\widetilde{H}_{19}(b)$ and $\widetilde{H}_{20}^{(2)}(a,b)$ of Hurst parameter $H\in[1/2,1)$. We multiply average errors by $10$ to make them visible.

\begin{table}[htbp]
\caption{Values of the  estimators for $H\in[1/2,3/4)$ and known scale coefficients}
\begin{center}
\begin{tabular}{ccccccccccc}
	\hline
	              $H$                & $0.5$  & $0.525$ & $0.55$ & $0.575$ & $0.6$  & $0.625$ & $0.65$ & $0.675$ &  $0.7$  & $0.725$ \\ \hline
	   $\mu\widetilde{H}_{19}(b)$    & $.4996$  & $.5253$  & $.549$ & $.5761$  & $.6002$ & $.6271$  & $.6334$  & $.6562$  & $.7324$ & $.7747$ \\ \hline
	$10\delta\widetilde{H}_{19}(b)$ & $.0147$ &  $.0184$  & $.0407$ & $.0512$  & $.0882$ & $.1322$  & $.2449$ & $.9517$  & $1.399$ & $2.348$ \\ \hline
	 $\mu \widehat{H}_{20}(a,b)$   & $.5$  &  $.525$   & $.55$  &  $.5753$   & $.5999$  &  $.6254$   & $.6493$  &  $.6766$   &  $.7079$  &  $.725$ \\ \hline
	$10\delta \widehat{H}_{20}(a,b)$ & $.0006$ & $.001$  & $.0018$ & $.0031$  & $.0038$ & $.0098$  & $.0177$ & $.0398$  & $.0943$  & $.1435$  \\ \hline
\end{tabular}
\end{center}
\end{table}
We see that $\widehat{H}_{20}(a,b)$ outperforms $\widetilde{H}_{19}(b)$ by a good margin, but the advantage of the latter is that it uses only knowledge of $b$. 
\subsection{Summary}
To facilitate the usage of the estimators, we summarize our findings about them. 

For $H\in(0,1/2)$, it is better to use the estimator $\widetilde H$ for the Hurst parameter. The estimator for the scale coefficient $a$ is quite reliable, but always overestimates the coefficient for $H\in(1/4,1/2)$. The estimator for $b$ is virtually useless.

For $H\in(1/2,3/4)$, there is no good estimator for the Hurst parameter. Only the regression estimator $\bar H^{(2)}$ is useful for values of $H$ between $0.55$ and $0.6$, but still the error is comparable with the length of this integral. The coefficient $b$ can be estimated efficiently, while the estimator for $a$ is useless. Nevertheless, it is possible to construct efficient estimators for $H$ using the knowledge of $b$ or of the both scale coefficients. 

Finally, for $H>3/4$, the estimation of $H$ is not possible (even the knowledge of the scale coefficients will not help). However, it is possible to distinguish statistically between the cases $H>3/4$ and $H<3/4$ by looking at the statistic $U^{H,2}_k$.

\bibliographystyle{klunamed}
\bibliography{mixedvars}
\end{document}